\documentclass[a4paper,USenglish,cleveref, autoref, thm-restate]{lipics-v2021}
\pdfoutput=1 %uncomment to ensure pdflatex processing (mandatatory e.g. to submit to arXiv)
\hideLIPIcs  %uncomment to remove references to LIPIcs series (logo, DOI, ...), e.g. when preparing a pre-final version to be uploaded to arXiv or another public repository
\nolinenumbers %uncomment to disable line numbering

\title{Orienting undirected phylogenetic networks to tree-child networks} %#TODO Please add

\titlerunning{Orienting undirected phylogenetic networks to tree-child networks} %#TODO optional, please use if title is longer than one line

\author{Shunsuke Maeda}{Department of Pure and Applied Mathematics, Graduate School of Fundamental Science and Engineering, Waseda University, Japan}{shunsuke.m0131@moegi.waseda.jp}{https://orcid.org/0009-0009-0616-8836}{}%#TODO mandatory, please use full name; only 1 author per \author macro; first two parameters are mandatory, other parameters can be empty. Please provide at least the name of the affiliation and the country. The full address is optional. Use additional curly braces to indicate the correct name splitting when the last name consists of multiple name parts.

\author{Yusuke Kaneko}{Department of Applied Mathematics, School of Fundamental Science and Engineering, Waseda University, Japan}{}{}{}

\author{Hideaki Muramatsu}{Department of Applied Mathematics, School of Fundamental Science and Engineering, Waseda University, Japan}{}{}{}

\author{Yukihiro Murakami}{Delft Institute of Applied Mathematics, Delft 
University of Technology, The 
Netherlands}{y.murakami@tudelft.nl}{https://orcid.org/0000-0003-1355-5884}{}

\author{Momoko Hayamizu\footnote{Corresponding author}}{Department of Applied Mathematics, Faculty of Science and Engineering, Waseda University, Japan}{hayamizu@waseda.jp}{https://orcid.org/0000-0001-8825-6331}{This work was supported by JST FOREST Program (Grant Number JPMJFR2135, Japan).}

\authorrunning{S. Maeda, Y. Kaneko, H. Muramatsu, Y. Murakami, and M. Hayamizu}
\Copyright{Shunsuke Maeda, Yusuke Kaneko, Hideaki Muramatsu, Yukihiro Murakami, and Momoko Hayamizu} %#TODO mandatory, please use full first names. LIPIcs license is "CC-BY";  http://creativecommons.org/licenses/by/3.0/

 %#TODO mandatory; please add comma-separated list of keywords

\category{} %optional, e.g. invited paper

\relatedversion{} %optional, e.g. full version hosted on arXiv, HAL, or other respository/website
\EventEditors{A\"{i}da Ouangraoua and Djamal Belazzougui}
\EventNoEds{2}
\EventLongTitle{23rd Workshop on Algorithms in Bioinformatics (WABI 2023)}
\EventShortTitle{WABI 2023}
\EventAcronym{WABI}
\EventYear{2023}
\EventDate{September 3-6, 2023}
\EventLocation{Houston, TX, United States}
\EventLogo{}
\SeriesVolume{}
\ArticleNo{}
\definecolor{mypink}{rgb}{0.9, 0.0, 0.4}
\definecolor{mygreen}{rgb}{0.0,0.5,0.0}
\definecolor{mypurple}{rgb}{0.6,0.0,0.6}
\definecolor{myfuchsia}{rgb}{0.9,0.0,0.9}
\definecolor{mybrown}{rgb}{0.7,0.3,0.3}
\definecolor{mygray}{rgb}{0.6,0.6,0.6}
\newcommand{\tmp}[1]{\textcolor{mygray}{#1}}

\usepackage{algpseudocode,algorithm} %added by MH
\usepackage{amsmath}
\renewcommand{\algorithmicrequire}{\textbf{Input:}} %Require->Input
\renewcommand{\algorithmicensure}{\textbf{Output:}} %Ensure->Output

\newtheorem{problem}[theorem]{Problem}

\usepackage{here}

\begin{document}

\maketitle

%#TODO mandatory: add short abstract of the document
\begin{abstract}
Phylogenetic networks are used to represent the evolutionary history of 
species. They are versatile when compared to traditional phylogenetic trees, as 
they capture more complex evolutionary events such as hybridization 
and horizontal gene transfer. 
Distance-based methods such as the Neighbor-Net algorithm are widely 
used to compute phylogenetic networks from data. However, the output is 
necessarily an undirected graph, 
posing a great challenge to deduce the direction of genetic 
flow in order to infer the true evolutionary history. 
Recently, Huber \emph{et al.} investigated two different computational problems 
relevant to orienting undirected phylogenetic networks into directed ones. In 
this paper, we consider the problem of orienting an undirected binary network 
into a tree-child network. We give some necessary conditions for determining 
the tree-child orientability, such as a tight upper bound on the size of 
tree-child orientable graphs, as well as many interesting examples. In 
addition, we introduce new families of undirected phylogenetic networks, the 
jellyfish graphs and ladder graphs, that are orientable but not tree-child 
orientable. We also prove that any ladder graph can be made tree-child 
orientable by adding extra leaves, and describe a simple algorithm for 
orienting a ladder graph to a tree-child network with the minimum number of 
extra leaves. We pose many open problems as well.
\end{abstract}

\section{Introduction}
Phylogenetics is a field that studies the evolutionary history and 
relationships among species. A widely used representation for these 
relationships is the phylogenetic tree, which shows the branching pattern of 
the evolutionary history. However, due to recombination and reticulate events, 
such as hybridization in plants and horizontal gene transfer in bacteria and 
viruses, the evolutionary relationships among species may not always follow a 
strictly tree-like pattern. In such cases, phylogenetic networks provide a more 
suitable representation~\cite{huson2010phylogenetic}.

The Neighbor-Net algorithm~\cite{bryant2004neighbor} is a popular approach to 
construct phylogenetic networks from distance data. It runs fast and always 
outputs a planar graph, making it easy to visualize biological data. However, 
the networks produced by this method are often complex and difficult to 
interpret, as they are undirected. In order to better understand the flow of 
genetic material among species, it is important to orient the edges of 
phylogenetic networks.

Despite its importance, the phylogenetic network orientation problem has 
received limited attention; conversely, the problem of orienting a phylogenetic 
tree has been investigated 
thoroughly~\cite{kinene2016rooting, tria2017phylogenetic, 
van2018unrooted}. In light of this, Huber~\textit{et al.}~\cite{orienting} 
recently proposed two types of problems for 
orienting undirected phylogenetic networks, along with solutions for each 
problem. The first problem called \textsc{Constrained Orientation} asks one to 
determine if an undirected phylogenetic network can be oriented as a directed 
phylogenetic network, given a distinguished edge (called the root) and all 
in-degrees of each vertex. The idea is to subdivide the distinguished edge by a 
vertex and orient the edges incident to it away from the vertex. Because the 
in-degrees of each vertex are specified, one can then orient the edges of the 
whole network (if certain conditions are satisfied).
%The first problem called \textsc{Constrained Orientation} asks the 
%following. Given an
%not-necessarily-binary 
%undirected phylogenetic 
%network $N$ together with a distinguished edge and desired in-degrees for each 
%vertex, does there exist an orientation $\hat{N}$ of $N$ so that~$\hat{N}$ is 
%a 
%phylogenetic network? 
%the resulting 
%digraph $\hat{N}$ is under the prescribed the constraints on the location of 
%its root and on the in-degree of each vertex. 
Huber~\textit{et al.}\;proved that 
such an orientation 
%$\hat{N}$ 
is unique if it exists and provided a linear-time 
algorithm for computing an orientation from a given network.
%$\hat{N}$ from $N$. 
The second problem called 
\textsc{$C$-Orientation} asks, given a binary undirected phylogenetic network 
$N$, whether there exists an orientation 
%$\hat{N}$ 
of $N$ such that the 
resulting directed graph becomes a network of a desired class of directed phylogenetic 
networks. Huber \textit{et al.} provided an FPT algorithm for solving the 
\textsc{$C$-Orientation} problem, and also proved that the problem is NP-hard 
in the case when $C$ is the class of so-called tree-based networks. Following 
this, Fischer and Francis~\cite{fischer2020tree} showed that one can 
characterize undirected tree-based networks as those that can be tree-base 
oriented.

In this paper, we discuss the \textsc{$C$-orientation} problem when $C$ is a 
so-called 
tree-child network, \textit{i.e.}, the problem of determining whether an 
undirected binary phylogenetic network can be rooted and oriented to be a 
tree-child network. Tree-child networks are prominent in the literature for its 
algorithmic ease of use (see, for example \cite{semple2016phylogenetic}), and as
they can be biologically motivated as a class of networks in which every 
ancestor passes on genetic material to an extant species via means of vertical 
descent~\cite{Transactions}. Here, we give necessary conditions for 
tree-child orientability in terms of number of edges as well as inter-leaf 
distances. Intuitively, our results imply that networks having too many edges, or 
networks in which leaves are pairwise too far apart cannot be tree-child 
oriented. Furthermore, we give two classes of undirected networks, called the 
jellyfish graph and the ladder graph. We show that jellyfish graphs cannot be 
tree-child oriented; we show that ladder graphs can only be tree-child oriented 
when the number of edges is small. For ladder graphs that are not tree-child 
orientable, we give a sharp lower bound on the number of leaves that must be 
added to make it tree-child orientable (Theorem~\ref{thm:ladder.algm}).

%In other related literature, the orientation problem is a classical problem in 
%graph theory, and it is formulated as giving an orientation to every edge in a 
%given graph to satisfy certain conditions.

%YM: I feel the following paragraph has no relevance to this paper.
%MH: Yes I agree, thanks. 
%It is interesting to study when one can create a tree-child network from 
%distance information. Bordewich and Semple~\cite{determingphylo.} considered 
%the problem of reconstructing an oriented phylogenetic network from the 
%path-length distances between taxa and showed that binary tree-child networks 
%are essentially determined by such distance information. This research was 
%followed by Bordewich~\textit{et al.}~\cite{constructing}, where they 
%developed 
%an efficient method for constructing a weighted tree-child network using the 
%root position and the multi-sets of up-down path distances between each pair 
%of 
%taxa under a certain assumption. 

%Also seems irrelevant.
%However, in this paper we are interested in the problem of orienting the 
%undirected graphs output by distance-based algorithms, such as the 
%Neighbor-Net 
%algorithm. In such situations, multisets of path distances are often 
%unavailable. \mh{introduction needs major revision}
%In addition, the Neighbor-Net algorithm always outputs planar graphs. Therefore, our main focus here is on the problem of determining whether a given planar graph can be rooted and oriented into a tree-child network.

The remainder of this paper is organized as follows. In 
Section~\ref{sec:Prelim}, we set up the basic definitions and notation in the 
field of graph theory and combinatorial phylogenetics. In 
Section~\ref{sec:Known}, we recall pertinent definitions and results regarding 
\textsc{Constrained Orientation} and \textsc{$C$-Orientation}. 
In Section~\ref{sec:necessary.conditions}, we give necessary conditions for the 
tree-child orientation problem. In Section~\ref{sec:Jelly.Ladder}, we give two 
classes of undirected networks, the jellyfish graph and the ladder graph, and 
explore their tree-child orientability. 
%In particular we give a sharp lower bound 
%for the number of leaves that need to be added to make a ladder graph 
%tree-child orientable 
In Section~\ref{sec:Discussion}, 
we discuss our results and provide open problems for future research.

\section{Definitions and Notation}\label{sec:Prelim}
\subsection{Graph theory}
In this paper, we will only consider connected, finite, simple graphs, which we now define. 
A \emph{graph} is an ordered pair $(V, E)$ of a set $V$ of \emph{vertices} and a set $E$ of \emph{edges} between vertices. 
Given a graph $G$, its vertex-set and edge-set are denoted by $V(G)$ and $E(G)$, respectively. 
A graph $G$ is said to be \emph{finite} if $V(G)$ and $E(G)$ are finite sets. 
A \emph{directed graph} is a graph where each edge has a direction associated with it, whereas an \emph{undirected graph} is a graph where no edge has a direction. 
An edge of directed graph that is oriented from vertex $u$ to vertex $v$ is denoted by $(u, v)$. 
An edge of undirected graph between vertices $u$ and $v$ is denoted by $(u, v)$ or $(v, u)$. 
A directed or undirected graph is \emph{simple} if it contains neither loop (i.e. edge that starts and ends at the same vertex $u$) nor multiple edges (i.e. different edges $(u, v)$, $(u', v')$ with $u=u'$, $v=v'$). 
%An undirected graph is said to be \emph{simple} if it is a graph without multiple edges (two or more edges between the same two vertices) or loops (an edge that connects a vertex to itself). 
Two simple graphs $G_1$ and $G_2$ are \emph{isomorphic} if there exists a bijective function $\phi:V(G_1)\rightarrow V(G_2)$ such that $(u, v)\in E(G_1)$ if and only if $(\phi(u), \phi(v))\in E(G_2)$. An undirected graph $G$ is called \emph{the underlying graph} of a directed graph $\hat{G}$ if the undirected graph obtained by ignoring the direction of each edge of $\hat{G}$ is isomorphic to $G$.

For a vertex $v$ of an undirected graph $G$,  the \emph{degree} of $v$ in $G$, denoted by $deg_G(v)$, is the number of edges incident to $v$. 
For a vertex $v$ of a directed graph $N$, the \emph{in-degree} of $v$in $N$, denoted by $indeg_N(v)$ or $d_N^-(v)$, is the number of edges incoming to $v$. Similarly, the \emph{out-degree} of $v$, denoted by $outdeg_N(v)$ or $d_N^+(v)$, is the number of edges outgoing from $v$. 
A \emph{directed path} is a directed graph $G$ that can be represented by an alternating sequence of vertices and consecutive edges $v_1, (v_1, v_2), v_2, \dots, (v_{k-1}, v_k), v_k$, where all vertices are distinct and we have $(indeg_G(v_1), outdeg_G(v_1))=(0,1)$, $(indeg_G(v_k), outdeg_G(v_k))=(1, 0)$ and $(indeg_G(v_i), outdeg_G(v_i))=(1,1)$ for any vertex $v_i$ other than $v_1, v_k$. 
A \emph{directed cycle} is a graph $C$ that satisfies the above conditions except that $v_1=v_k$ and we have  $indeg_C(v_i)=outdeg_C(v_i)=1$ for any vertex $v_i$.
A directed graph containing no directed cycle is  \emph{acyclic}. 

The concepts of path, cycle, and acyclic graphs are defined similarly for undirected graphs. An undirected graph is \emph{connected} if there exists a path between any pair of vertices, whereas a directed graph $G$ is \emph{(weakly) connected} if the underlying graph of $G$ is connected.
Given two vertices $u, v$ of a connected undirected graph $G$, the \emph{distance} between $u$ and $v$, denoted by $d_G(u, v)$,  is the number of the edges in the shortest path between them.

%A \emph{planar graph} is a graph that can be drawn in the plane without any edges crossing each other. A \emph{plane graph} is a planar graph with a fixed drawing in the plane. The \emph{faces} of a plane graph are the regions bounded by the edges, and the \emph{outer face} is the unbounded region. The well-known Euler's formula says that the number of faces of is independent of the embedding we chose and thus an invariant of the planar graph itself. 

\subsection{Phylogenetic networks}
Throughout the paper, $X$ denotes a finite set with $|X|\geq 2$ that can be biologically interpreted as a set of present-day species. The set $X$ is often referred to as a ``label set'' because each element of $X$ is used to label a vertex of a graph. Recall that all networks considered here are connected, finite, simple graphs. 

An \emph{undirected binary phylogenetic network $N$ (on $X$)} is a directed acyclic graph such that 
i) $X=\{v\in V(N)\mid deg_N(v)=1\}$ and ii) for any $v\in V(N)\setminus X$, $deg_N(v)=3$.  The vertices in $X$ are called \emph{leaves} of $N$.  
We call an undirected phylogenetic network \emph{non-binary} if it is not necessarily binary, i.e. if the above condition ii) is generalized as follows: for any non-leaf vertex $v$, $deg_N(v) \geq 3$. Note that we have defined undirected phylogenetic networks in such a way that they do not have degree-two vertices, simply by convention. It is possible to allow a finite number of degree-two vertices to exist, because their existence is trivial when  orienting undirected phylogenetic networks (the formal definitions related to orientation will be given in Section~\ref{sec:Known}).

A \emph{directed binary phylogenetic network} $\hat{N}$ on $X$ is a undirected acyclic graph such that i) $X=\{v\in V(N)\mid (indeg_N(v), outdeg_N(v))=(1, 0)\}$;  
ii) there exists a unique vertex $\rho\in V(\hat{N})$ with $(indeg_{\hat{N}}(\rho), outdeg_{\hat{N}}(\rho)=(0, 2)$; 
iii) for any $v\in V(\hat{N})\setminus (X\cup \{\rho\})$, $(indeg_{\hat{N}}(v), outdeg_{\hat{N}}(v))\in \{(1, 2), (2, 1)\}$. 
The vertex $\rho$ is called \emph{the root} of $\hat{N}$, and the vertices in $X$ are called \emph{leaves} of $\hat{N}$. Each vertex of in-degree two is called a \emph{reticulation}, while each non-leaf vertex of in-degree one is called a \emph{tree vertex}. 
We call a directed phylogenetic network \emph{non-binary} if it is not necessarily binary, i.e. if the condition iii) is generalized as follows: for any non-leaf non-root vertex $v$, $indeg_N(v)+ outdeg_N(v) \geq 3$. The same remark regarding the definition of undirected non-binary networks also applies to the definition of directed non-binary networks. 
%if each non-root, non-leaf vertex is allowed to have in-degree two or more and out-degree two or more. 
%Vertices with in-degree at least 2 are reticulations, while vertices with in-degree 1 are tree vertices.

Given an edge $(u,v)$ of a directed phylogenetic network $\hat{N}$, $u$ is a \emph{parent} of $v$, and $v$ is a \emph{child} of $u$.
A directed binary phylogenetic network $\hat{N}$ is called a \emph{tree-child network} if each non-leaf vertex of $\hat{N}$ is a parent of either a tree vertex or a leaf of $\hat{N}$ \cite{Comparison}. See Figure~\ref{fig:tree-child} for examples. 

\begin{figure}[hbt]
  \centering
  \includegraphics[width=0.8\textwidth]{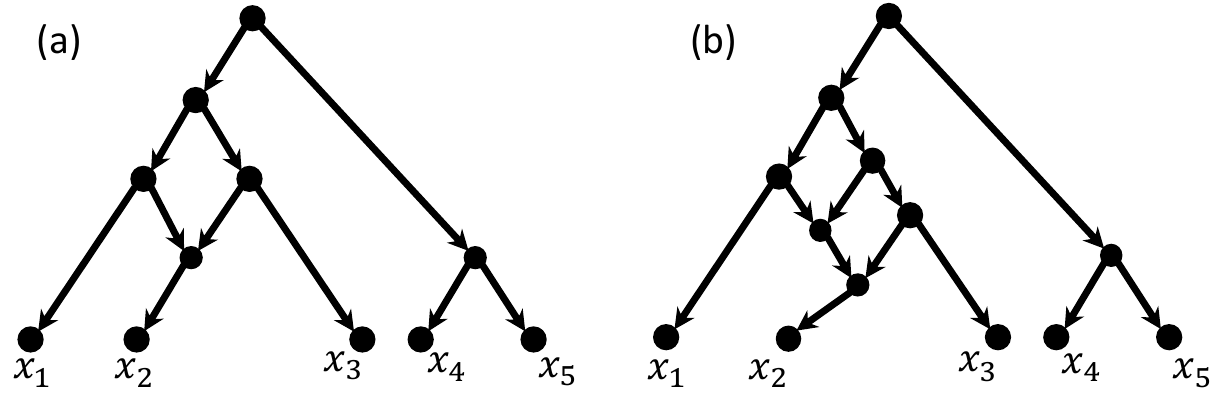}
 \caption{Examples of directed binary phylogenetic networks on $X=\{x_1, x_2, x_3, x_4, x_5\}$. The network shown in (a) is a tree-child network. The network shown in (b) is not a tree-child network. 
  \label{fig:tree-child}}
\end{figure}

\section{Known results on orienting phylogenetic networks}\label{sec:Known}
Orienting undirected phylogenetic networks could find many biological applications in the future, but currently its mathematical and computational aspects are still at an early stage of research. In this section, we briefly recall the necessary definitions and summarize relevant results from Huber~\textit{et al.}~\cite{orienting}, and then describe the problem we will consider in this paper.

\subsection{Terminology: rooting and orienting phylogenetic networks}
In the usual context of graph theory, given an undirected graph $G$, ``orienting'' $G$ typically refers to the operation of creating a directed graph $\hat{G}$ by orienting each edge of $G$, and $\hat{G}$ is called an ``orientation'' of $G$. However, we must stress that, when it comes to an undirected phylogenetic network, the term ``orienting'' may refer to a different operation. More precisely, given a (not necessarily binary) undirected phylogenetic network $N$, Huber~\textit{et al.}~\cite{orienting} defined the operation of \emph{orienting} $N$ as the following procedure (see Figure~\ref{fig:degcut4}(a)(b) and Figure~\ref{fig:TCO}(a)(b) for an illustration): 1) choose a unique edge $e_\rho$ of $N$ called a \emph{root edge}; 2) insert the root $\rho$ into $e_\rho$ (the resulting graph is denoted by $N_\rho$); 3) assign a direction to each edge of $N_\rho$.  
 In other words, orienting an undirected phylogenetic network $N$ means orienting the root-inserted graph $N_\rho$, not of the original graph $N$.

We say that a directed phylogenetic network $\hat{N}$ is an \emph{orientation} of an undirected phylogenetic network $N$ if it is possible to obtain $\hat{N}$ by orienting $N$.  If $N$ is a tree, choosing a root edge $e_\rho$ automatically determines the directions of all edges of $N_\rho$, but in general, rooting $N$ does not specify an orientation of $\hat{N}$. Therefore, to avoid ambiguity, we modify the terminology in Huber~\textit{et al.}~\cite{orienting}. In this paper, we will use the term \emph{rooting} $N$ to mean the process of selecting $e_\rho$ and creating $N_\rho$ from $N$, and \emph{orienting} $N$ to mean the entire process of rooting $N$ and then assigning a direction to each edge of $N_\rho$. Strictly speaking, the rooting step can be seen as yielding a partially directed graph $N_\rho^\prime$, because the root $\rho$ automatically determines the direction of the two edges starting from $\rho$; however, the only difference between $N_\rho$ and $N_\rho^\prime$ is whether the two edges are assigned the obvious directions or not, so in this paper we use the same symbol $N_\rho$ without distinguishing between them.

\subsection{Orientation constrained by the root edge and in-degrees}
The \textsc{Constrained Orientation} is an orientation problem under the constraints of the position of the root edge $e_{\rho}$ and the desired $indeg_{\hat{N}}(v)$ of each vertex (see also Figure~\ref{fig:degcut4}).

\begin{problem}\label{prob:constrained.orientation}
(\textsc{Constrained Orientation})
\begin{description}
\item[\textbf{Input:}]
An undirected non-binary phylogenetic network $N=(V,E)$ on $X$, an edge $e_{\rho}\in E$ into which a unique root $\rho$ is inserted, and a labeling map $\delta_{N}^-:V\rightarrow \mathbb{N}$ that specifies $indeg_{\hat{N}}(v)$ for each vertex $v\in V$.
\item[\textbf{Output:}]
An orientation $\hat{N}$ of $(N,e_{\rho}, \delta_{N}^-)$ if it exists, and NO otherwise.
\end{description}
\end{problem}

Huber~\textit{et al.}~\cite{orienting} introduced the notion of a degree cut, which is the key ingredient for characterizing orientability. The notion of a degree cut is illustrated in Figure~\ref{fig:degcut4}(d). 

\begin{definition}\label{def:degree-cut}
Let $N = (V, E)$ be an undirected non-binary phylogenetic network on $X$ with $e_\rho \in E$ a distinguished edge, and let $N_\rho = (V_\rho, E_\rho, X)$ be the graph obtained from $N$ by subdividing $e_\rho$ by a new vertex $\rho$. Given the desired in-degree $d_{N}^-(v)$ of each vertex $v \in V$, a \emph{degree cut} for $(N, e_\rho, d_{N}^-)$ is an ordered pair $(V', E')$ with $V' \subseteq V$ and $E' \subseteq E_\rho$ such that the following hold in $N_\rho$:
\begin{enumerate}
    \item $E'$ is an edge cut of $N_\rho$; i.e. its removal makes $N_\rho$ into a disconnected graph $N_\rho \setminus E'$ that consists of two or more connected components;
	\item $\rho$ is not in the same connected component of $N_\rho \setminus E'$ as any $v\in V'$;
    \item each edge in $E'$ is incident to exactly one element of $V'$; and
    \item each vertex $v \in V'$ is incident to at least one and at most $d_{N}^-(v) - 1$ edges in $E'$.
\end{enumerate}
In particular, when $N$ is binary, a degree cut is called a \emph{reticulation cut}. 
\end{definition}
%We note that when $N$ is binary (i.e. for each vertex, the desired in-degree is either one or two), the third and fourth conditions in Definition~\ref{def:degree-cut} imply $|E'|=|R'|$, where $R$ is the set of vertices $v$ with $d_{N}^-(v)=2$ (i.e. reticulations). 

%In \cite{orienting}, it was shown that a solution $\hat{N}$  to the \textsc{Constrained Orientation} problem is uniquely determined for  $(N,e_{\rho}, \delta_{N})$ if it exists, and that the solution can be computed in linear time. In addition, they provided a characterisation for when there exists such a feasible orientation $\hat{N}$ under the constraints on the root edge $e_\rho$ and on the in-degree labeling $\delta_N$. 

\begin{theorem}[Theorems 1 and 2 in \cite{orienting}] \label{thm:huber1-2}
%\begin{theorem}[Theorem 1 in \cite{orienting}] \label{thm:huber1}
%Let $N = (V, E, X)$ be an undirected non-binary phylogenetic network, $e_\rho \in E$ be a distinguished edge, and $d_{N}^-(v)$ be the desired in-degree of each vertex $v \in V$, where $1 \leq d_{N}^-(v) \leq d_{N}(v)$. Then $(N, e_\rho, d_{N}^-)$ is orientable if and only if $(N, e_\rho, d_{N}^-)$ has no degree cut and $\sum_{v \in V} d_{N}^-(v) = |E| + 1$.
%\end{theorem}
%
%\begin{theorem}[Theorem 2 in \cite{orienting}] \label{thm:huber2}
%Let $N = (V, E, X)$ be an undirected non-binary phylogenetic network, $e_\rho \in E$ be a distinguished edge, and $d_{N}^-(v)$ be the desired in-degree of each vertex $v \in V$, where $1 \leq d_{N}^-(v) \leq d_{N}(v)$. Then Algorithm 1 decides whether $(N, e_\rho, d_{N}^-)$ is orientable, in which case, it finds an orientation in time $O(|E|)$. Moreover, this orientation is the unique orientation of $(N, e_\rho, d_{N}^-)$.
%\end{theorem}
%
%INTEGRATION OF THE ABOVE TWO
%Let $N = (V, E, X)$ be an undirected non-binary phylogenetic network, $e_\rho \in E$ be a distinguished edge, and $d_{N}^-(v)$ be the desired in-degree of each vertex $v \in V$, where $1 \leq d_{N}^-(v) \leq d_{N}(v)$. Then $(N, e_\rho, d_{N}^-)$ is orientable if and only if $(N, e_\rho, d_{N}^-)$ has no degree cut and $\sum_{v \in V} d_{N}^-(v) = |E| + 1$. Algorithm 1 in \cite{orienting} decides whether $(N, e_\rho, d_{N}^-)$ is orientable, and finds an orientation in time $O(|E|)$ if it exists. Moreover, this orientation is the unique orientation of $(N, e_\rho, d_{N}^-)$.
%*******
%Momoko's modified version
Let $N = (V, E, X)$ be an undirected non-binary phylogenetic network, $e_\rho \in E$ be a distinguished edge, and $d_{N}^-(v)$ be the desired in-degree of each vertex $v \in V$, where $1 \leq d_{N}^-(v) \leq d_{N}(v)$. Then, the following statements hold. 
\begin{enumerate}
	\item $N$ has an orientation $\hat{N}$ that satisfy the constraints $(e_\rho, d_{N}^-)$ if and only if $(N, e_\rho, d_{N}^-)$ has no degree cut and $\sum_{v \in V} d_{N}^-(v) = |E| + 1$. (For binary $N$, $\sum_{v \in V} d_{N}^-(v) = |E| + 1$ is equivalent to $|R|=|E|-|V|+1$, where $R$ is the set of vertices $v$ with $d_{N}^-(v)=2$.)
	\item  Algorithm 1 in \cite{orienting} decides whether $\hat{N}$ exists, and finds $\hat{N}$ if it exists both in $O(|E|)$ time. 
	\item If it exists, $\hat{N}$ is the unique orientation of $N$ under the constraints $(e_\rho, d_{N}^-)$. 
\end{enumerate}
\end{theorem}

%\begin{theorem}[\cite{orienting}] 
%\tmp{An orientation $\hat{N}$ of $(N,e_{\rho}, \delta_{N})$  exists if and only if no degree-cut exists.}	
%\end{theorem}
%However, although there is a theorem for determining whether the inputs $e_{\rho}$ and $indeg_{N}(v)$ are orientable, there is no simple rule for how to specify $e_{\rho}$ and $indeg_{N}(v)$ to be orientable.\mh{??}

%On the other hand, the following problem is also considered as a different constrained orientation problem\cite{orienting}.

\begin{figure}[hbt]
  \includegraphics[width=14cm]{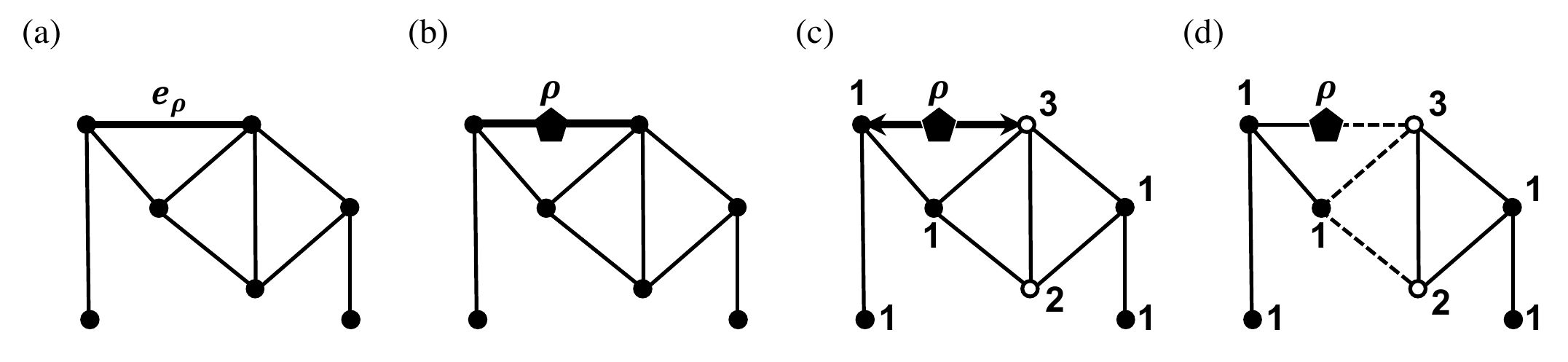}
  \centering
  \caption{An illustration of the Constraint Orientation Problem. (a) An undirected phylogenetic network $N$ with a distinguished root edge $e_\rho$. (b) A phylogenetic network $N_\rho$ obtained by inserting the root vertex $\rho$ into $e_\rho$. (c) A phylogenetic network $N$ with the constraints $(e_\rho, d_N^-)$, where the root edge $e_\rho$ is chosen as in (a), and each vertex $v$ is labeled by the desired in-degree $d_N^-(v)$.  (d) If we let $V'$ be the set of white vertices and $E'$ be the set of dashed-line edges, then $(V', E')$ is a degree cut for $(N, e_\rho, d_N^-)$ (Definition~\ref{def:degree-cut}). Then, the first statement in Theorem~\ref{thm:huber1-2} tells us that there exists no feasible orientation of $N$ under the constraint $(e_\rho, d_N^-)$. 
  \label{fig:degcut4}}
\end{figure}

\subsection{Orientation to a desired class $C$ of networks}
Problem~\ref{prob:constrained.orientation} was the problem of orienting an undirected graph under constraints on the position of the root edge and on the desired in-degree of each vertex. The next problem, called \textsc{$C$-Orientation}, is an orientation problem with different constraints. This problem asks whether a given graph can be oriented to be a directed phylogenetic network belonging to a desired class $C$, where the position of the root edge and the degree of entry of each vertex are unknown.
Huber~\textit{et al.}~\cite{orienting} considered the $C$-orientation problem under the assumption that the input graph $N$ is binary, unlike Problem~\ref{prob:constrained.orientation}.

\begin{problem}\label{prob:C-orientation}
(\textsc{$C$-Orientation})
\begin{description}
   \item[INPUT:]\mbox{}
            An undirected binary phylogenetic network $N$ on $X$, and a class $C$ of directed binary phylogenetic network on $X$.  
   \item[OUTPUT:]\mbox{}
	    A $C$-orientation $\hat{N}$ of $N$ if it exists, and NO otherwise.
\end{description}
\end{problem}

As there are $|E|$ ways to choose the root edge and $\binom{|V|}{|R|}$ ways to choose the $|R|$ reticulations among $|V|$ vertices, they described a simple exponential time to determine whether $N$ is $C$-orientable or not, by checking all possible cases using the linear-time algorithm for solving Problem~\ref{prob:constrained.orientation}. They gave an FPT algorithm for a particular $C$ satisfying several conditions, and also proved that the $C$-orientation problem is NP-complete when $C$ is the class of tree-based networks. However, it is still not yet well understood for which class $C$ the $C$-orientation problem is NP-complete or solvable in polynomial time. %A non-trivial $C$ has not yet been found for which the C-orientation problem can be solved in polynomial time.

We must emphasize that if the class $C^\prime$ is a subclass of $C$, this does not implies that $C^\prime$-orientation is easier or harder than $C$-orientation.
% Then, $C$-orientation is interesting either for a large class $C$ of phylogenetic networks or a small $C$ class of phylogenetic networks. 
This motivates us to study the following problem.

\begin{problem}\label{prob:tree-child-orientation}
(\textsc{Tree-child Orientation})
\begin{description}
   \item[INPUT:]\mbox{}
            An undirected binary phylogenetic network $N$ on $X$ and the class $C$ of tree-child networks on $X$.
   \item[OUTPUT:]\mbox{}
	    A tree-child orientation $\hat{N}$ of $N$ if it exists, and NO otherwise.
\end{description}
\end{problem}

We say that an undirected binary phylogenetic network $N$ is \emph{tree-child orientable} if $\hat{N}$ in Problem~\ref{prob:tree-child-orientation} exists and call such an orientation $\hat{N}$ a \emph{tree-child orientation} of $N$ (see  Figure \ref{fig:TCO}).

%\textsc{$C$-Orientation} has as output a directed phylogenetic network belonging to class $C$ that satisfies certain properties, and it was shown that the solution can be computed by the FPT-algorithm\cite{orienting}.
%However, a good class $C$ example that can be oriented in polynomial time has not yet been found.

%In this paper, we narrow down the discussion to the class $C$ in \textsc{$C$-Orientation} as tree-child networks and consider the following problem.

%
%\begin{figure}[hbt]
%\centering
%  \includegraphics[width=0.55\textwidth]{tree-child-orientation.pdf}
%  \caption{Example of tree-child orientable network. The undirected phylogenetic network in (a) is tree-child orientable as shown in (b). White circles are reticulations and the black pentagon is a new vertex introduced as the root.
%  \label{fig:tree-child-orientable.not-orientable}}
%\end{figure}
%
%\begin{figure}[hbt]
%\centering
%\includegraphics[width=0.5\textwidth]{nottree-child-orientable-pro5.pdf}
%  \caption{Examples of undirected graphs that are orientable to some phylogenetic networks but not tree-child orientable. 
%  \label{fig:orientable.but}}
%\end{figure}

\begin{figure}[hbt]
\centering
\includegraphics[width=0.8\textwidth]{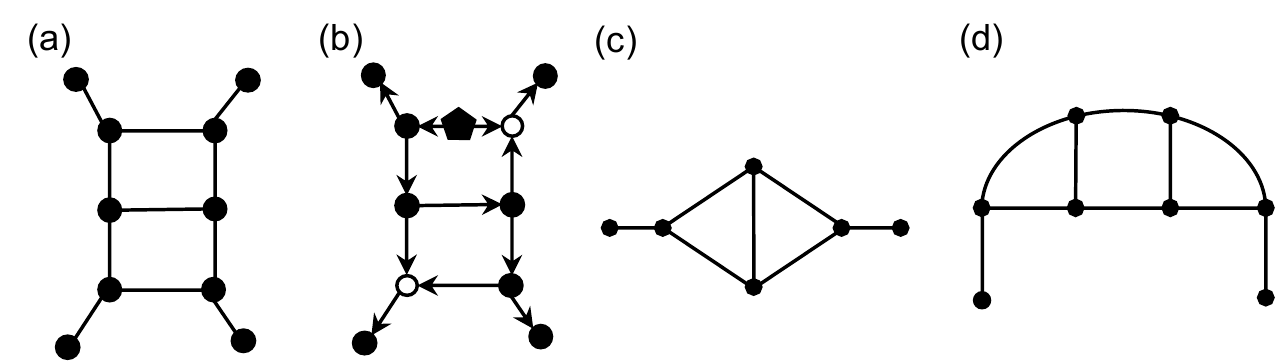}
  \caption{(a) An example of a tree-child orientable network $N$. (b) A tree-child orientation $\hat{N}$ of $N$, where the black pentagon is the  root $\rho$ and white circles are reticulations. (c),(d) Examples of undirected graphs that are orientable to some phylogenetic networks but not tree-child orientable. 
   \label{fig:TCO}}
\end{figure}

\section{Necessary Conditions for Tree-Child Orientability}\label{sec:necessary.conditions}
There is no previous study that has focused on Problem~\ref{prob:tree-child-orientation} so far. 
In this section, we will give some necessary conditions to ensure that $N$ is tree-child orientable. 
%These conditions will be given with respect to the number of edges of $G$ and with respect to the distances between all pairs of leaves of $G$. 

\subsection{The number of edges of tree-child orientable graphs}

We recall the following useful result together with its proof from \cite{CountingPhylogeneticNetworks} as it clarifies the relationships between the numbers of reticulations, tree vertices, leaves, and edges of directed binary phylogenetic networks.

\begin{proposition}[{\cite[Lemma 2.1]{CountingPhylogeneticNetworks}}]\label{relation}
%Let $N$ be a directed binary phylogenetic network with $n$ vertices that consist of $\ell$ leaves, $r$ reticulation vertices, and $t$ tree vertices. Then $t = \ell+r-2$ and $n = 2t+3$. Also, $N$ has $3r+2\ell-2$ edges.
Let $N=(V, A)$ be any directed binary phylogenetic network on $X$ that has $|X|$ leaves, $r$ reticulations, $t$ tree vertices and the root $\rho$ with $outdeg(\rho)=2$. Then, we have  $t = |X|+r-2$,  $|V| = 2t+3$, and $|A|=3r+2|X|-2$. 
%\tmp{Here it is assumed that the root $\rho$ satisfies $(indeg(\rho), outdeg(\rho))=(0,2)$. When $(indeg(\rho), outdeg(\rho))=(0,1)$, $|A|$ becomes greater by one. Note that in this setting the root $\rho$ is allowed to have a reticulation vertex.}
\end{proposition}
\begin{proof}
	Note first that $|V| = r+|X|+t+1$. The hand-shaking lemma for directed graphs states that the sum of the out-degrees equals the number $|A|$ of edges which, in turn, equals the sum of the in-degrees. Thus, we have $r+2t +2 = |A| = 2r+t +|X|$, which yields $t = r+|X|-2$. The other equations follow easily. This completes the proof. 
\end{proof}

\begin{proposition}[{\cite[Proposition 1]{Comparison}}]\label{leaves-reticulations}
%    If a tree-child network has $n$ leaves and $r$ reticulations, then $r\leq n-1$.
For a directed binary phylogenetic network $N$ on $X$ with $r$ reticulations, if $N$ is a tree-child network, then $r\leq |X|-1$ holds.
\end{proposition}

%\subsection{General orientability}
\begin{lemma}\label{root.children}
If $\hat{N}$ is a directed binary phylogenetic network on $X$, then,  at least one of the children of the root $\rho$ of $\hat{N}$ is a tree vertex. %neither the two reticulations is not neighborhood of the root. 
\end{lemma}
\begin{proof}
By Theorem 3 in Huber~\textit{et al.}~\cite{orienting}, 
if both children of $\rho$ were reticulations, there would exist a reticulation cut as shown in Figure~\ref{fig-Neither-of-the-two-is-reticulation}. 
\end{proof}
\begin{figure}[hbt]
  \includegraphics[width=0.25\textwidth]{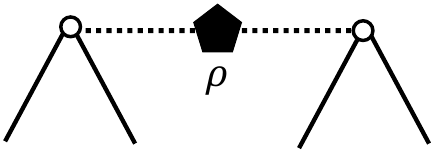}
  \centering
  \caption{Proof of Lemma~\ref{root.children}. The black vertex is the root and the white circles are reticulations. The set of the two edges in dotted line, together with the set of the two reticulations, forms a ``reticulation cut'', which is not allowed to exist in orientable networks by Theorem 3 in Huber~\textit{et al.}~\cite{orienting}.
  %The figure shows neither of the two is reticulation in directed binary phylogenetic network
  \label{fig-Neither-of-the-two-is-reticulation}}
\end{figure}

%\begin{lemma}\label{lem:leaves-reticulations}
%%    If a tree-child network has $n$ leaves and $r$ reticulations, then $r\leq n-1$.
%Let  $\hat{N}$ be a tree-child network on $X$ with $r$ reticulations.  If $\hat{N}$ is triangle-free (i.e. there exist no three edges $(v_1, v_2), (v_1, v_3), (v_2, v_3)$), then $r\leq |X|-2$ holds.
%\end{lemma}
%\begin{proof}
%	Let $N^*$ be a tree-child network on $X$ with $|X|$ reticulations i.e. maximally reticulated. Let $p, q$ be the children of the root of $N^*$. Then either $p$ or $q$ is a reticulation (if both are tree vertices then one could create a tree-child network on $X$ with $|X|+1$ reticulations by introducing an additional directed edge between the edges $(\rho, p)$ and $(\rho, q)$.  However, ****.  Therefore, if  $\hat{N}$ is a triangle-free tree-child network, then $\hat{N}$ is not maximally reticulated.  Thus, $r\leq |X|-2$ holds. This completes the proof.
%\end{proof}

\begin{theorem}\label{edges.leq.5X-6}
If a binary undirected phylogenetic network $N=(V,E)$ on $X$ is tree-child orientable, then $|E|\leq 5|X|-6$ holds. Moreover, this upper bound is tight. 
\end{theorem}
\begin{proof}
If $\hat{N}=(\hat{V}, \hat{E})$ is a tree-child orientation of $N$, then $|\hat{E}|=|E|+1$ holds because orienting $N$ to $\hat{N}$ involves the operation of inserting a root into an edge of $N$, i.e. subdividing an edge of $N$ exactly once. The proof will be completed if we can show that  $|\hat{E}|\leq 5|X|-5$.
Let $r$ be the number of reticulations of $\hat{N}$. 
Then, Proposition~\ref{relation} gives $|\hat{E}|=2|X|+3r-2$,  and  Proposition~\ref{leaves-reticulations} gives $r\leq |X|-1$. 
Thus, we obtain $|\hat{E}|\leq 2|X|+3(|X|-1)-2=5|X|-5$. The network in Figure \ref{fig:tightsize} ensures the tightness of this bound.  This completes the proof. 
%\maeda{$|E|\leq 5|X|-6$}.
%Suppose that the undirected phylogenetic network is orientable to a tree-child network $N=(V,A)$ on $X$, there exists a tree-child network $\hat{N}=(\hat{V},\hat{E},\hat{X})$ which is the orientation of $N$.
%In the process of orienting, we subdivide one edge of $N$ and insert a root, so $|\hat{E}|=|E|+1$ holds.
\end{proof}
\begin{figure}[hbt]
\centering
  \includegraphics[width=0.5\textwidth]{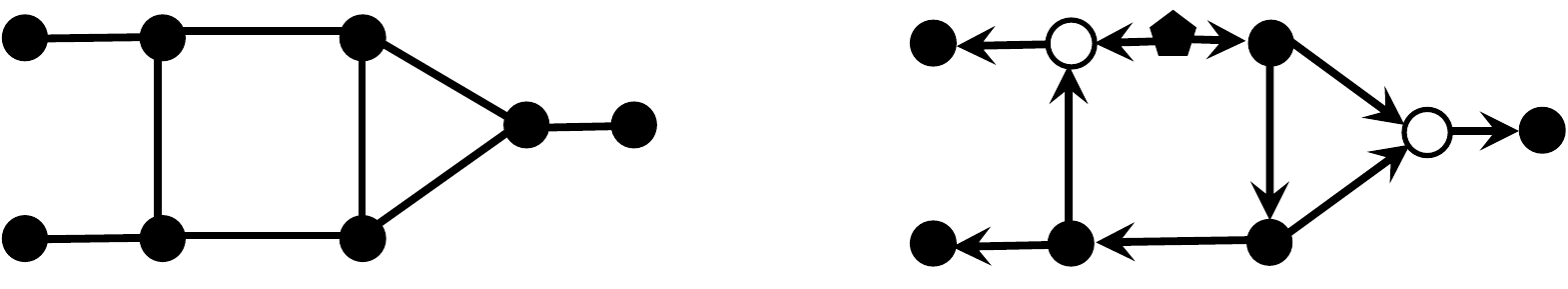}
  \caption{An example that shows the bound in Theorem~\ref{edges.leq.5X-6} is 
  tight. Left: a binary phylogenetic network with $|X|=3$ leaves and $5|X|-6=9$ 
  edges. Right: a tree-child orientation of the network (open circles represent 
  reticulation vertices). 
 \label{fig:tightsize}}
\end{figure}

The condition given in Theorem~\ref{edges.leq.5X-6} is not a sufficient condition but it is practically useful to check the tree-child orientability. For example, we can immediately see that the two graphs in Figure~\ref{fig:TCO}(c)(d) are not tree-child orientable.

\subsection{The shortest-path distances between leaves and implications for appropriate rooting}

Tree-child networks are a subclass of cherry-picking networks (also known as orchard networks) that were introduced in~\cite{erdHos2019class, janssen2021cherry}.  Therefore, tree-child orientable networks need to be orchard-orientable. For the reader's convenience, we now recall the definitions of cherries and reticulated cherries from~\cite{determingphylo.} as follows (see Figure~\ref{fig:cherry.and.reticulated.cherry} for an illustration). 

Let $\hat{N}$ be a directed binary phylogenetic network.
A \emph{cherry} of $\hat{N}$ is a pair $\{x_i, x_j\}$ of leaves such that the parent $x_i^\prime$ of $x_i$ and the parent $x_j^\prime$ of $x_j$ are the same vertex of $\hat{N}$. 
A \emph{reticulated cherry} of $\hat{N}$ is a pair $\{x_i, x_j\}$ of leaves such that there exists an undirected path with two internal vertices between $x_i$ and $x_j$ in the underlying graph $N$, exactly one of which is a reticulation. 

\begin{figure}[hbt]
\centering
  \includegraphics[width=0.3\textwidth]{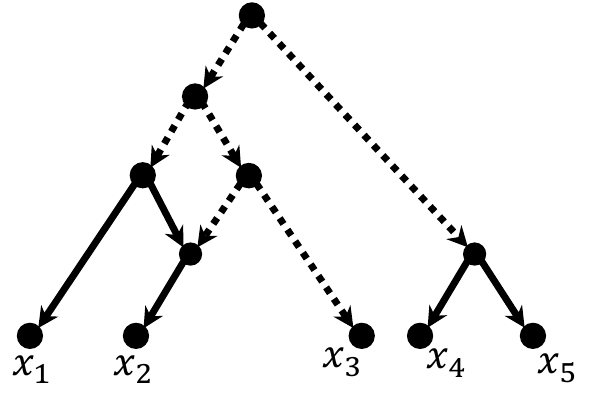}
  \caption{An illustration of a cherry and reticulated cherry (solid lines). In the above network, $\{(x_1, x_2)\}$ is a reticulated cherry and  $\{(x_4, x_5)\}$ is a cherry.
 \label{fig:cherry.and.reticulated.cherry}}
\end{figure}

\begin{proposition}[{\cite[Lemma 4.1]{determingphylo.}}]\label{TC.has.cherry}
    Let $N$ be a binary tree-child network on $X$. If $|X|\geq 2$, then $N$ contains either a cherry or a reticulated cherry.
\end{proposition}

From the above lemma, we obtain the following result. 

\begin{proposition}\label{2or3}
    Let $N$ be an undirected phylogenetic network on $X$. If $N$ does not have leaves $x, x^\prime\in X$ with either $d_N(x, x^\prime)=2$ or $d_N(x, x^\prime)=3$, then $N$ is not tree-child orientable. 
    %not orchard-orientable (and therefore not tree-child orientable).
\end{proposition}
\begin{proof}
By Proposition~\ref{TC.has.cherry}, tree-child networks have a pair of leaves which is called cherry or reticulated cherry. This means that if $N$ does not have a pair of leaves that are at a distance of 2 or 3 from each other, then $N$ is not tree-child orientable. This completes the proof. 
\end{proof}

The condition in Proposition~\ref{2or3} is useful for appropriate rooting for tree-child orientation. For example, given the undirected graph in Figure~\ref{fig:not-tree-child-orientable1}, one can easily choose a right root edge $e_\rho$ for tree-child orientation using the distance condition.

% Applying Proposition~\ref{2or3}, we see that there are edges in an undirected phylogenetic network for which no roots should be inserted (Figure~\ref{fig:not-tree-child-orientable1}). it suggests that the position of the root insertion has an effect on the possibility of orienting an undirected phylogenetic network to a tree-child network.

\begin{figure}[hbt]
\centering
  \includegraphics[width=0.6\textwidth]{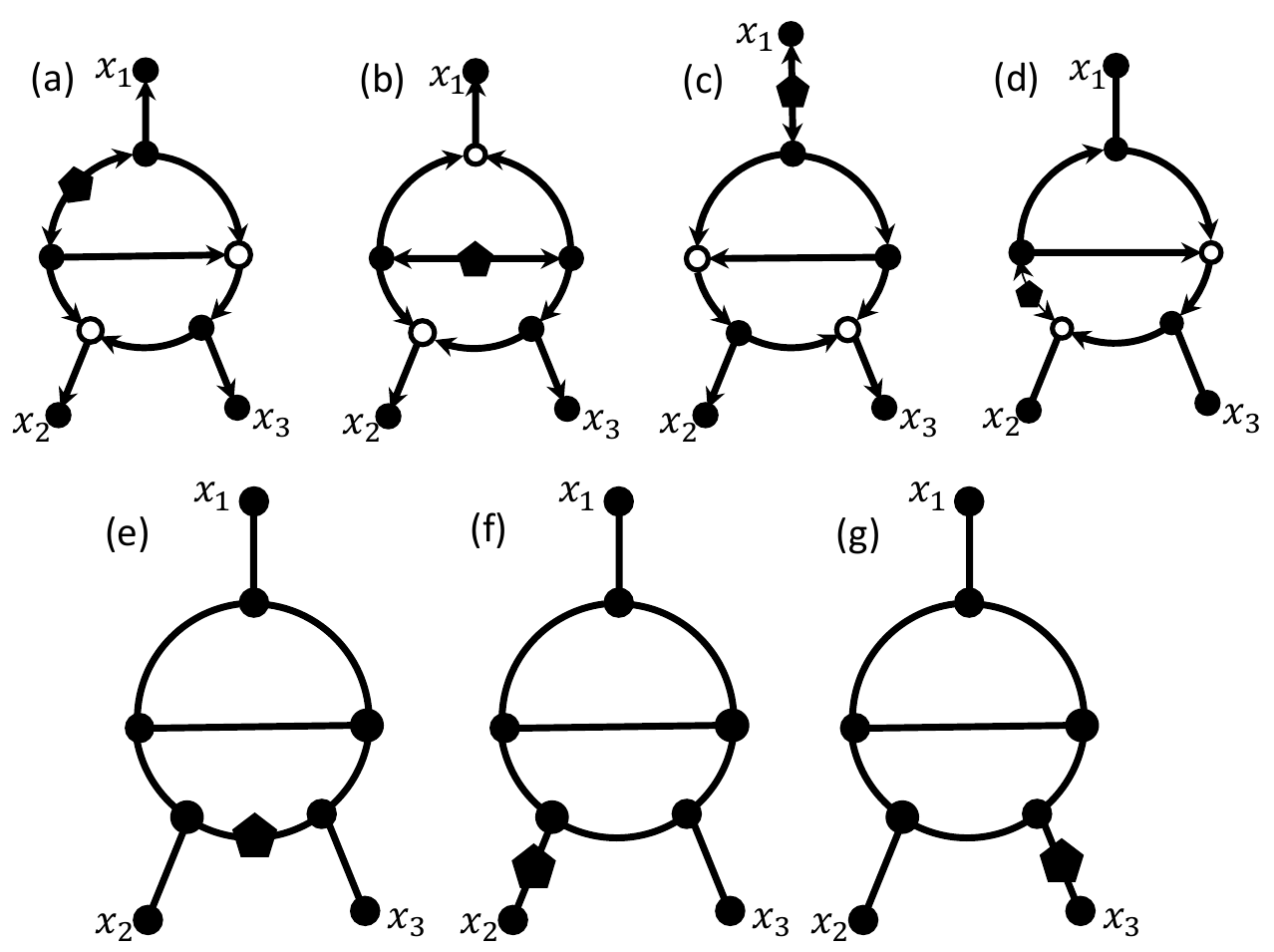}
  \caption{The possible choices (up to symmetry) for the root edge $e_\rho$ for a tree-child orientable graph. It is possible to create a tree-child network if the root is introduced as in (a)--(d), but not if the root is placed as in (e)--(g). 
  %\mh{Maeda-san sotsuron p18 fig 4.4. Give all examples both good and bad rooting.}
  \label{fig:not-tree-child-orientable1}}
\end{figure}

\subsection{The number of reticulation vertices}
Even if we want to orient an undirected binary phylogenetic network $N$ on $X$ to any directed binary phylogenetic network $\hat{N}$ on $X$ without considering the tree-child property of $\hat{N}$, we must carefully determine which vertices of $N$ are to be reticulations in $\hat{N}$. 
Fortunately, however, the number $|R|$ of reticulations in $\hat{N}$ is automatically determined by $N$ and can be easily computed. The reticulation number $R$ can be expressed by different formulas (for example, Proposition~\ref{relation} gives $|R|=\frac{1}{3}(|E|-2|X|+3)$), but the following is more widely known. This can be proved by different proofs but we omit the proof because it is included in the first condition of Theorem~\ref{thm:huber1-2}.

\begin{proposition}\label{betti.number}
    For any undirected binary phylogenetic network $N=(V,E)$ on $X$ and for any orientation $\hat{N}$ of $N$,  the number $|R|$ of  reticulations in $\hat{N}$ is given by $|R|=|E|-|V|+1$. 
\end{proposition}

Note that $|E|-|V|+1$ is the minimum number of edges that need to be removed from a connected undirected graph $G=(V, E)$ to make $G$ a tree. This quantity is referred to using different terminology such as the circuit rank, the dimension of the cycle space, the first Betti number, and so on.  

\section{Undirected binary phylogenetic networks that are orientable but not 
tree-child orientable}\label{sec:Jelly.Ladder}
%\maeda{The following is an undirected network like a jelly fish that is not tree-child orientable and the reason why this is not orientable.}
%At first, we have to note the Lemma\maeda{Necessary and sufficient conditions for orientation}, and the following corollary introduced.
In this Section, we introduce classes of undirected phylogenetic networks that are orientable but not tree-child orientable.

\subsection{The jellyfish graph}
\emph{The jellyfish graph} $J_k$  ($k\geq 1$) is an undirected binary 
phylogenetic network on $X = \{x_1,\dots,x_k\}$ with $2k+6$ vertices and $2k+9$ 
edges, where the neighbor $x'_i$ of each leaf $x_i$ forms a path $x'_1,\dots, 
x'_k$ as described in Figure~\ref{fig:J1-J5}. 
We will show that the jellyfish graph is orientable but not tree-child orientable.

\begin{figure}[hbt]
\centering
\includegraphics[width=0.6\textwidth]{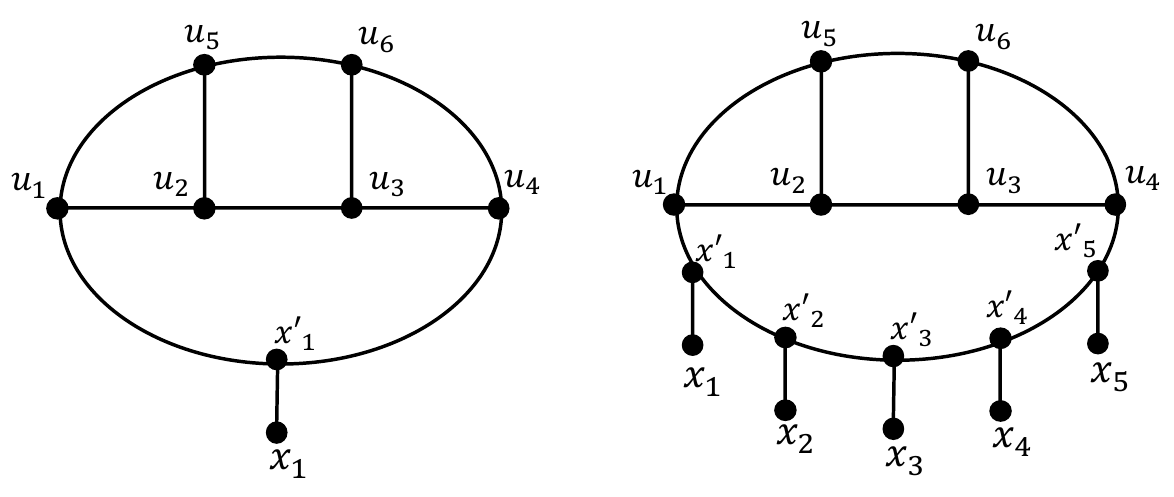}
%#TODO revise caption and ref
  \caption{Left: The jellyfish graph $J_1$ with one leaf. Right: The jellyfish graph $J_5$ with five leaves.  
  %a jellyfish graph. We note that \tmp{both satisfy the necessary conditions in Theorem~\ref{edges.leq.5X-6} and in  Proposition~\ref{2or3}}.\maeda{Check if this ladder is not-tree-child-orientable. Include Fig3.7right in Muramatsu's thesis.)}
  %\label{fig:J1-J5-Ladder}
  \label{fig:J1-J5}
  }
\end{figure}

\begin{lemma}\label{jellyfish:at.most.one.ret}
Let $J_k$ be a jellyfish graph with a leaf-set $X=\{x_1,\dots, x_k\}$ and the set $X^\prime:=\{x_1^\prime,\dots, x_k^\prime\}$ of neighbors of the leaves. Then, 
 $X^\prime$ contains at most one reticulation in any orientation $\hat{J}_k$ of $J_k$.
 % such that $\hat{J}_k$ is a directed binary phylogenetic network on $X$, 
%	In the jellyfish graph, there exists at most one reticulation among the parents $x_1^\prime,\dots, x_k^\prime$ of the leaves $x_1, \dots , x_5$, regardless of the position of the root edge $e_\rho$. 
\end{lemma}
\begin{proof}
We use $J_5$ in Figure~\ref{fig:jellyfish.proof.rooting} but the proof is the same for any $k$. 
Suppose the root edge $e_\rho$ is neither a pendant edge $(x_i, x_i^\prime)$ nor an edge $(x_i^\prime, x_{i+1}^\prime)$ between pendant edges, as shown in the left of Figure~\ref{fig:jellyfish.proof.rooting}. Any directed path from $\rho$ to a leaf $x_i$ must include either the edge incoming to $x^\prime_1$ or the edge incoming to $x^\prime_5$. Suppose $x^\prime_3$ is a reticulation as described in the figure. Then, since the graph is binary, this specifies the direction of each solid-line edge, making any $x^\prime_i$ other than $x^\prime_3$ a vertex with out-degree 2. Thus, we can conclude that there exists at most one reticulation among $x^\prime_i$'s (white circle). When the root is inserted in a pendant edge $(x_i, x_i^\prime)$ or an edge $(x_i^\prime, x_{i+1}^\prime)$ as described in the right of Figure~\ref{fig:jellyfish.proof.rooting}, we obtain the same conclusion by a similar argument. This completes the proof. 
\end{proof}

\begin{theorem}\label{jellyfish.not.tc.orientable}
	The jellyfish graph $J_k$ is not tree-child orientable for any $k\in \mathbb{N}$.  
\end{theorem}

\begin{proof}
We use $J_5$ in Figure~\ref{fig:jellyfish.proof.rooting} but the proof is the same for any $k$. 
Let $\hat{J}_k$ be any orientation of $J_k$. 
By Proposition~\ref{betti.number},  $\hat{J}_k$ has exactly four reticulations. 
Let $X$ and  $X^\prime$ be as in Lemma~\ref{jellyfish:at.most.one.ret}. 
Lemma~\ref{jellyfish:at.most.one.ret} allows us to focus on the following two cases. Case 1: When $X^\prime$ contains exactly one reticulation as in Figure~\ref{fig:jellyfish.proof.reticulation1}. Case 2:  When $X^\prime$ contains no reticulation as in Figure~\ref{fig:jellyfish.proof.reticulation0}. Taking the symmetry of the jellyfish graph into account, we see that all the possibilities are listed in these figures. Recall that we cannot insert the root $\rho$ in an edge whose ends are both reticulations (Lemma~\ref{root.children}). As can be verified easily, regardless of the choice of the root edge $e_\rho$, each option ends up with having a forbidden configuration of tree-child networks (i.e. adjacent reticulations or a tree vertex with two reticulation children). Thus, the jellyfish graph is orientable but not tree-child orientable. This completes the proof. 
\end{proof}

\begin{figure}[hbt]
  \includegraphics[width=0.6\textwidth]{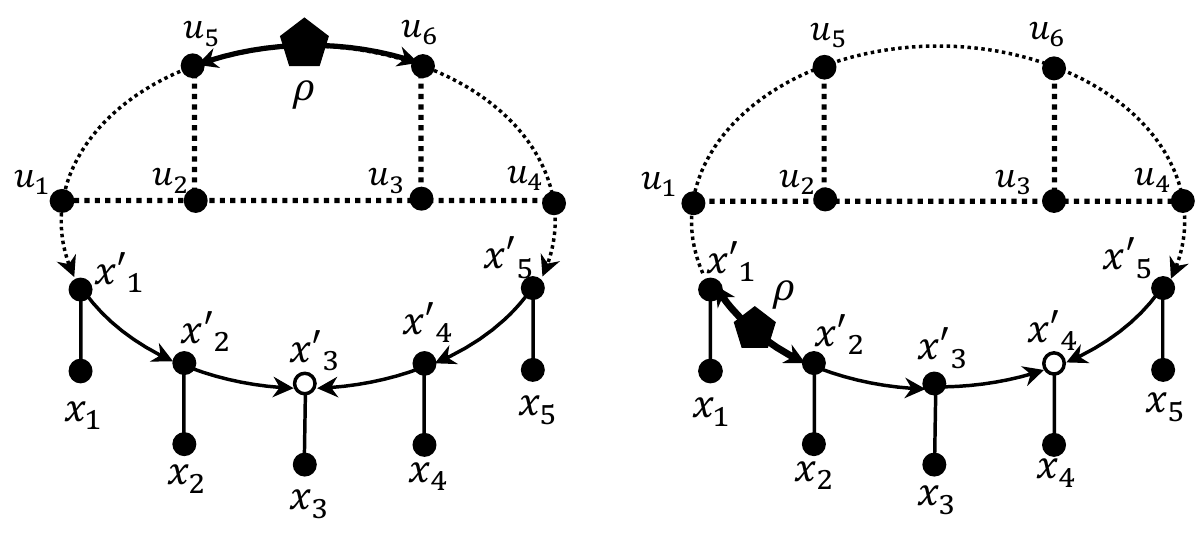}
  \centering
  \caption{Proof of Lemma~\ref{jellyfish:at.most.one.ret}. White vertices indicate reticulations. 
  %Left: if the root is inserted outsider of the solid line, the solid line contains only one reticulation (white circle). Right:  if the root is inserted in the solid line, the solid line contains at most one reticulation (white circle). In any case, the solid line can not contain two or more reticulations. 
  \label{fig:jellyfish.proof.rooting}}
\end{figure}

\begin{figure}[hbt]
  \includegraphics[width=\textwidth]{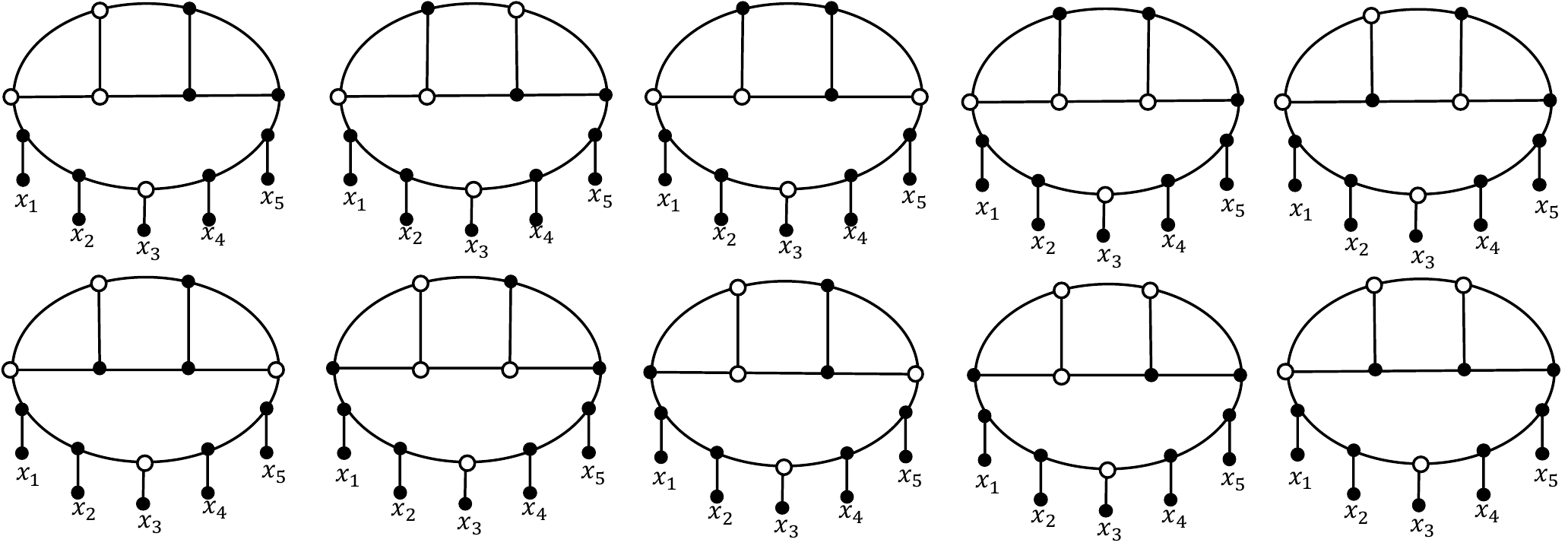}
  \centering
  \caption{Proof of Proposition~\ref{jellyfish.not.tc.orientable} (Case 1). 
  %The case when the solid line contains exactly one reticulation. 
%  We cannot insert the root between reticulation vertices (otherwise this produces a reticulation cut). 
  \label{fig:jellyfish.proof.reticulation1}}
\end{figure}

\begin{figure}[hbt]
  \includegraphics[width=0.8\textwidth]{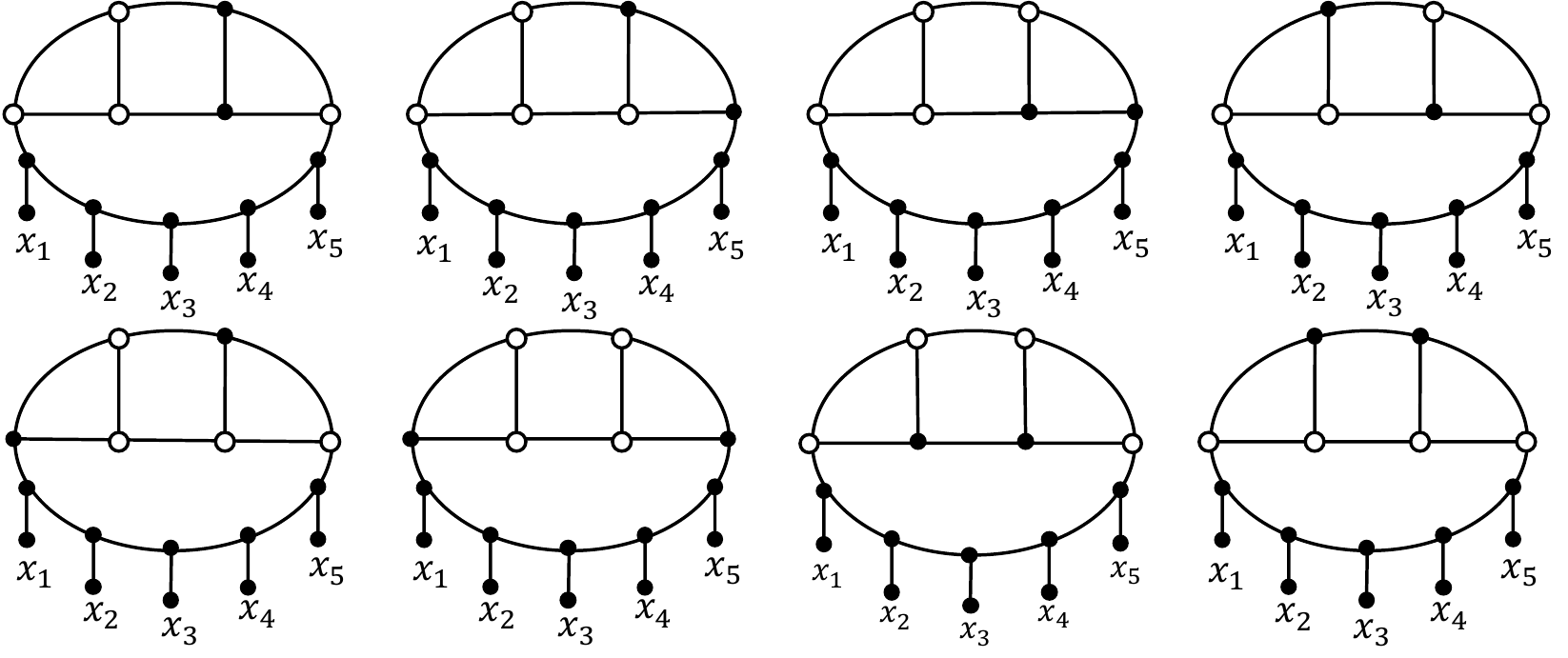}
  \centering
  \caption{Proof of Proposition~\ref{jellyfish.not.tc.orientable} (Case 2). 
  % The case when the solid line contains exactly one reticulation.  We cannot insert the root between reticulation vertices (otherwise this produces a reticulation cut).
  \label{fig:jellyfish.proof.reticulation0}}
\end{figure}
%figure

\subsection{The ladder graph}
\emph{A ladder graph} $L_{k}$ is the undirected binary phylogenetic network on 
$X = \{a,b,c,d\}$ with vertices~$\{t_i,b_i: i\in[k+1] = \{1,2,\ldots, k+1\}\}$ 
with paths~$a\, t_1 \ldots t_{k+1}\, c$ and $b\, b_1\, \ldots b_{k+1}\, d$, 
together with edges $\{(t_i, b_i): i\in [k]\}$ (see 
Figure~\ref{fig:ladder.convert} (a) for a ladder graph~$L_5$).
% as described in Figure~\ref{fig:ladder.convert}(a), 
Observe that $k= |E|-|V|+1$.
%#TODO define ladder more properly
% where $r$ denotes the reticulation number $|E|-|V|+1$. 
%We call the graph on the right in Figure~\ref{fig:J1-J5-Ladder} the ``ladder'' graph. 
It is easy to check that  $L_k$ (with 4 leaves) is tree-child orientable if and 
only if $k\leq 3$. However, as shown in Figure~\ref{fig:ladder.convert} (c)(d), 
it is possible to make  $L_k$ with $k\geq 4$ tree-child orientable by adding 
extra leaves at appropriate places. 
In this section, we will show that any $L_k$ with $k\geq 4$ can be converted into a tree-child orientable ladder graph by adding extra leaves at appropriate places. 
%Adding an extra leaf to a tree-child orientable graph still leaves it tree-child orientable. 

%
%\begin{proposition}\label{ladder.statement}
%The ladder graph is orientable but tree-child orientable.	
%\end{proposition}

\begin{figure}[hbt]
\includegraphics[width=0.9\textwidth]{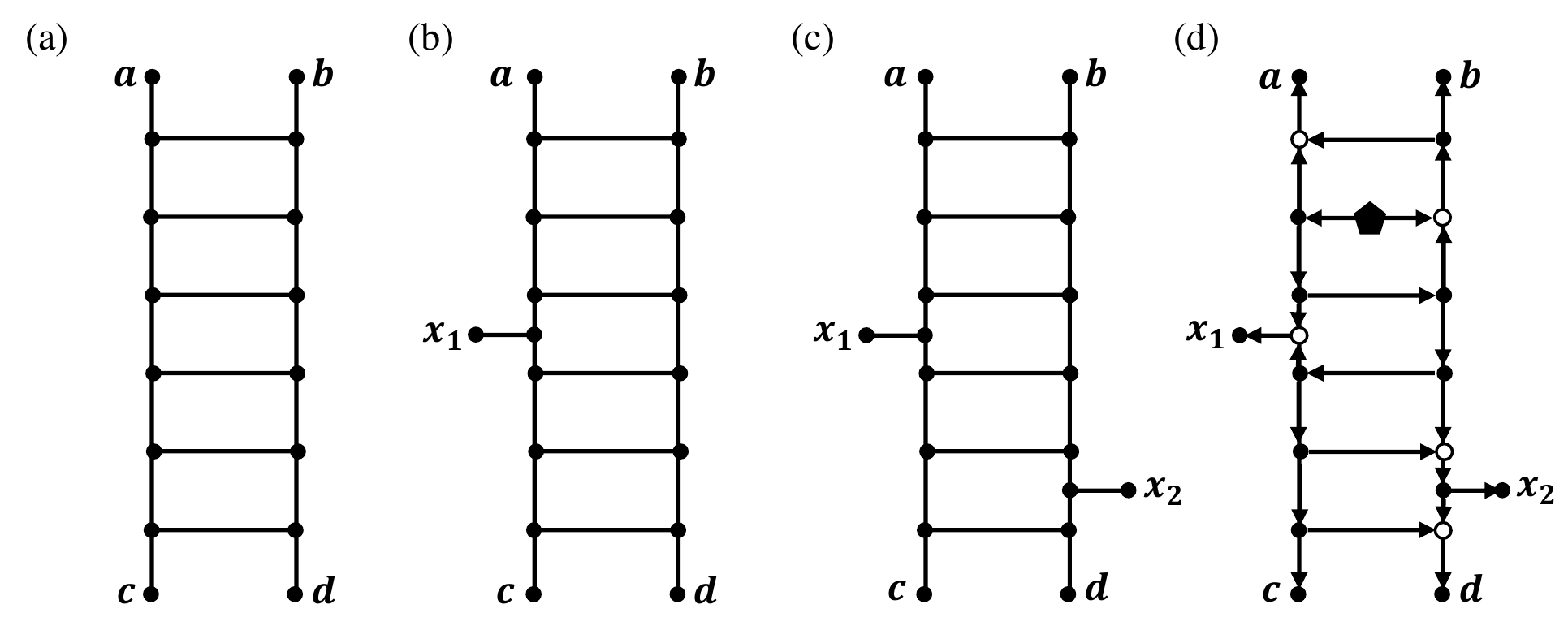}
  \centering
  \caption{(a) A ladder graph $L_5$ with five cycles, which is not tree-child orientable by Theorem~\ref{edges.leq.5X-6}. (b) A ladder obtained from $L_5$ by adding one extra leaf $x_1$, which is still not tree-child orientable. (c)  A ladder obtained by adding two extra leaves $x_1$ and $x_2$, which is tree-child orientable as shown in (d).  
  \label{fig:ladder.convert}}
\end{figure}

\begin{theorem}\label{thm:ladder.algm}
	Let $L_k$ be a ladder graph on leaves $\{a,b,c,d\}$ as described before. 
	Then, one can construct a tree-child orientable network $N$ from $L_k$ by 
	adding exactly $k-3$ leaves, but not by adding $k-4$ or fewer leaves.
\end{theorem}
\begin{proof}
We describe an algorithm for constructing a tree-child orientable network $N$ 
by adding exactly $k-3$ extra leaves to $N$. See Figure~\ref{fig:ladder.algm} 
for the construction.

%PROOF BY YUKI: START %MODIFIED BY MOMOKO
	Set the root edge $e_\rho := (t_{k-2}, b_{k-2})$. The idea is to place 
	$k-4$ leaves to one side of $e_\rho$ and $1$ leaf to the other side of 
	$e_\rho$. As for reticulations, we place $k-3$ reticulations to one side, 
	$2$ reticulations to the other side, and one on $e_\rho$ itself (i.e. 
	exactly one of $t_{k-2}$ and $b_{k-2}$ will be a reticulation in the 
	orientation).
	Add leaf $x_1$ to edge $(b_1,b_2)$. 
	Add leaf $x_{2j}$ to edge $(t_{2j+1}, t_{2j+2})$ for $j = 1, 2, \ldots, \lceil\frac{k-3}{2}-1\rceil$.
Add leaf $x_{2j+1}$ to edge $(b_{2j+2}, b_{2j+3})$ for $j = 1, 2, \ldots, \lfloor\frac{k-3}{2} - 1\rfloor$.	
	If $k$ is odd, add a leaf~$x_{k-3}$ to edge~$(b_k, b_{k+1})$ and if~$k$ is even, add a leaf~$x_{k-3}$ to edge~$(t_k, t_{k+1})$. 
	In total, we have added $k-3$ leaves to $L_k$. 
	To show the resulting graph $N$ is tree-child orientable, we now specify $k$ vertices to be reticulations, and obtain a tree-child orientation $\hat{N}$. 
	Let~$p_{x_i}$ denote the neighbour of leaf~$x_i$ for each $i\in [k-3]$.
	 We let $p_{x_1}$ be a tree vertex and let~$p_{x_i}$ be a reticulation for $i = 2,\ldots, k-4$. 
	We let~$b_1, b_2$ be reticulations. If~$k$ is odd, let~$t_{k-2}, b_k, 
	b_{k+1}$ 
%	(or alternatively $t_{k-2}, t_k, b_{k+1}$) 
	be reticulations; if 
	$k$ is even, let~$b_{k-2}, t_k, t_{k+1}$ be reticulations.
	%(or alternatively $b_{k-2}, b_k, t_{k+1}$). YM: No, this won't work as 
	%$b_{k-1}$ will have two reticulation children.
This allocation of the root and reticulations yields an orientation $\hat{N}$ 
of the ladder graph $N$ (and $\hat{N}$ is unique for this allocation by 
Theorem~\ref{thm:huber1-2}). To see that $\hat{N}$ is a tree-child network, 
observe that no two reticulations are adjacent and no two reticulations at 
distance $2$ have a common parent in $\hat{N}$. Theorem~\ref{edges.leq.5X-6} 
implies that at least $k-3$ additional leaves are necessary to obtain a 
tree-child orientable network. This completes the proof.
 %PROOF BY YUKI: END %MODIFIED BY MOMOKO
\end{proof}

\begin{figure}[hbt]
\includegraphics[width=0.9\textwidth]{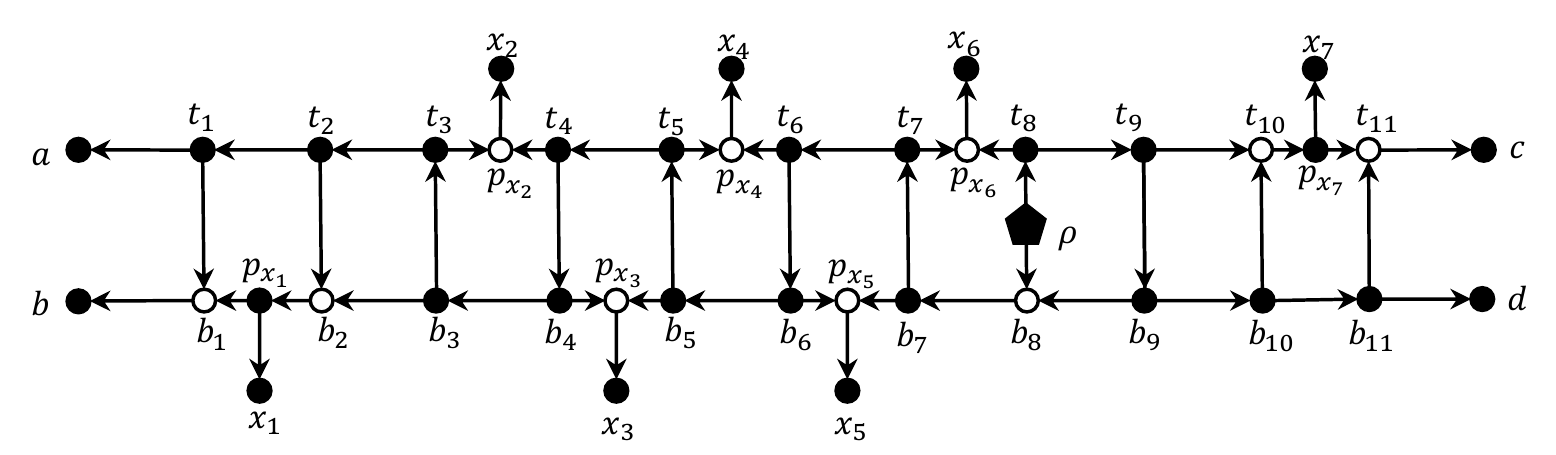}
  \centering
  \caption{Proof of Theorem~\ref{thm:ladder.algm}. An illustration for the case of $k=10$. The open circles represent the vertices specified as reticulations.
  \label{fig:ladder.algm}}
\end{figure}

We note that, unlike the case shown in Figure~\ref{fig:ladder.convert} (d), the ladder graph in Figure~\ref{fig:L5-not.TC.orientable} is not tree-child orientable although it satisfies the necessary conditions for the tree-child orientability described in Theorem~\ref{edges.leq.5X-6} and in  Proposition~\ref{2or3}. 

\begin{figure}[hbt]
\centering
\includegraphics[width=0.3\textwidth]{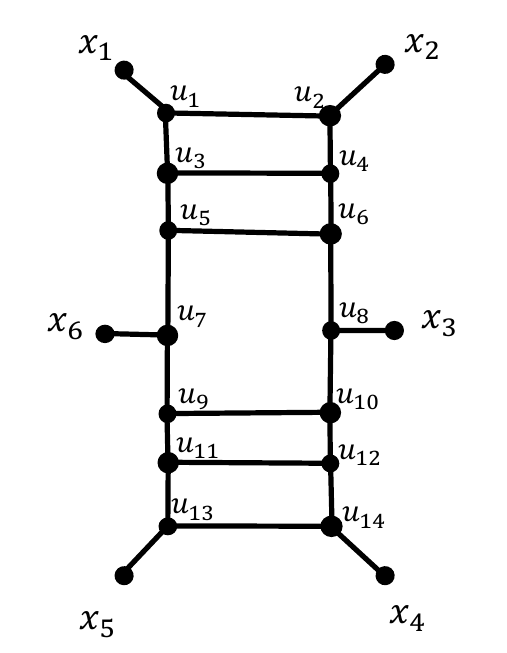}
%#TODO revise caption and ref
  \caption{An example of a ladder graph with additional leaves that is not tree-child orientable. If the position of the leaf $x_3$ is changed from the edge $(u_6, u_{10})$ to $(u_{12}, u_{14})$, it becomes tree-child orientable (as shown in Figure~\ref{fig:ladder.convert} (d)). 
  %with reticulation number $|R|=5$ with $|X|=6$ leaves. 
  \label{fig:L5-not.TC.orientable}
  }
\end{figure}

\section{Conclusion and open problems}\label{sec:Discussion}
%\subsection{Characterisation of tree-child orientable networks}
In this paper, we have discussed the Problem~\ref{prob:tree-child-orientation} (\textsc{Tree-child Orientation}), which asks, given an undirected phylogenetic network $N$, whether it is possible to orient $N$ to a tree-child network $\hat{N}$. To the best of our knowledge, this is the first study to focus on the tree-child orientation problem. While a characterization of tree-child orientable graphs is remains open, we have obtained several necessary conditions for the tree-child orientability.  In particular, Theorem~\ref{edges.leq.5X-6} provides a tight upper bound on the size of tree-child orientable graphs. In addition, based on the recent work of  Huber~\textit{et al.}~\cite{orienting} on Problem~\ref{prob:constrained.orientation} (\textsc{Constrained Orientation}), we introduced new families of undirected phylogenetic networks, the jellyfish graphs and ladder graphs, which are orientable but not tree-child orientable. We also proved that any ladder graph can be made tree-child orientable by adding extra leaves, and described an algorithm for converting a ladder graph into a tree-child network with the minimum number of extra leaves.

While we conjecture that \textsc{Tree-child Orientation} is NP-complete, there are many other interesting directions for future research. 
What is a necessary and sufficient condition for tree-child orientable networks? This question is still  interesting even if we restrict our attention to planar graphs, since a tree-child network is not necessarily planar (see the network in Figure~\ref{fig:non-planar.tree-child}). So which planar graphs are tree-child orientable? It is easy to see that undirected non-binary phylogenetic cactuses (defined in \cite{hayamizu2020recognizing}) are  tree-child orientable, but what about a more general subclass of planar graphs (e.g. outer-labeled planar graphs that are constructed by the Neighbor-Net algorithm)?

%As \cite{francis2018new} that defined three indices that measure the extent to which an arbitrary phylogenetic network deviates from being tree-based.
%

%We note that the condition in Theorem~\ref{edges.leq.5X-6} is a necessary condition but is not a sufficient condition for ensuring tree-child orientability as the examples in Figure~\ref{fig:not-tree-child-orientable1} demonstrate. 

\begin{figure}[hbt]
  \includegraphics[width=0.4\textwidth]{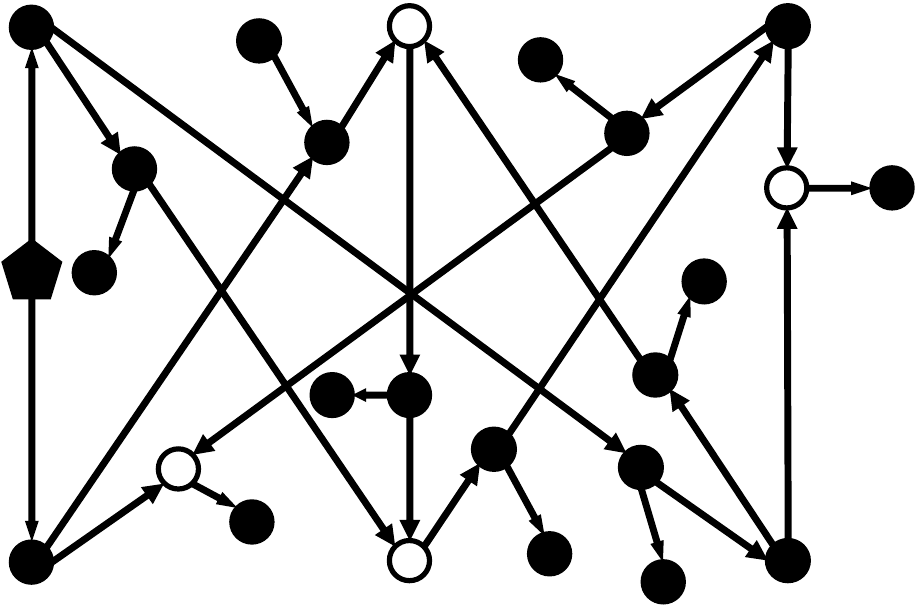}
  \centering
  \caption{An example of a non-planar tree-child phylogenetic network. Indeed, the graph is a directed binary tree-child network (i.e. each non-leaf vertex  has at least one child that is a tree vertex) and also contains a subgraph that is homeomorphic to $K_{3,3}$. By Kuratowski's Theorem~\cite{kuratowski1930probleme}, the graph is not planar. 
  \label{fig:non-planar.tree-child}}
\end{figure}

\bibliographystyle{plain}
\bibliography{lipics-v2019-sample-article.bib}

\end{document}